\documentclass[12pt]{article}
\usepackage{amsthm,amsmath,amssymb,verbatim,fancyhdr}
\usepackage{array,url}
\usepackage{fullpage}

\newtheorem{thm}{Theorem}[section]
\newtheorem{lem}[thm]{Lemma}
\newtheorem{cor}[thm]{Corollary}
\newtheorem{conj}[thm]{Conjecture}
\newtheorem{clm}{Claim}
\newtheorem{fact}{Fact}

\def\VEC#1#2#3{#1_{#2},\ldots,#1_{#3}}

\def\DX{\vec{\chi}}
\newcommand{\di}{\displaystyle}
\def\qed{\ifhmode\unskip\nobreak\hfill$\Box$\bigskip\fi \ifmmode\eqno{Box}\fi}

\newcommand{\RR}{\mathbb{R}}

\begin{document}

\title{Dichromatic number and fractional chromatic number}

\author{
  Bojan Mohar\thanks{Supported in part by an NSERC Discovery Grant (Canada),
   by the Canada Research Chair program, and by the
    Research Grant P1--0297 of ARRS (Slovenia).}~\thanks{On leave from:
    IMFM \& FMF, Department of Mathematics, University of Ljubljana, Ljubljana,
    Slovenia.}\\[1mm]
  Department of Mathematics\\
  Simon Fraser University\\
  Burnaby, BC, Canada\\
  {\tt mohar@sfu.ca}
\and
 Hehui Wu\thanks{This work was done while the author was a PIMS Postdoctoral Fellow at the Department of Mathematics, Simon Fraser University, Burnaby, B.C.}\\[1mm]
 Department of Mathematics\\
  University of Mississipi\\
  Oxford, MS\\
  {\tt hhwu@olemiss.edu}
}

\maketitle

\begin{abstract}
The dichromatic number of a graph $G$ is the maximum integer $k$ such that there exists an orientation of the edges of $G$ such that for every partition of the vertices into fewer than $k$ parts, at least one of the parts must contain a directed cycle under this orientation.
In 1979, Erd\H{o}s and Neumann-Lara conjectured that if the dichromatic number of a graph is bounded, so is its chromatic number. We make the first significant progress on this conjecture by proving a fractional version of the conjecture. While our result uses stronger assumption about the fractional chromatic number, it also gives a much stronger conclusion: If the fractional chromatic number of a graph is at least $t$, then the fractional version of the dichromatic number of the graph is at least $\tfrac{1}{4}t/\log_2(2et^2)$. This bound is best possible up to a small constant factor. Several related results of independent interest are given.
\end{abstract}

\section{Introduction}

For an undirected graph $G$, the \emph{chromatic number} $\chi(G)$ is the minimum number of independent sets whose union is $V(G)$. For a directed graph (digraph) $D$, the analogue to independent sets are acyclic vertex-sets, where we call a vertex-set \emph{acyclic} if it does not contain a directed cycle. Then, the \emph{chromatic number} $\chi(D)$ of $D$ is the minimum number of acyclic vertex-sets whose union is $V(G)$ (see \cite{NL82,BFJKM2004}). The \emph{dichromatic number} of an undirected graph $G$, denoted by $\vec{\chi}(G)$, is the maximum chromatic number over all its orientations \cite{ENL,NL82}.

In the late 1970's, Erd\H{o}s and Neumann-Lara \cite{ENL} posed the following conjecture:

\begin{conj}[Erd\H{o}s and Neumann-Lara \cite{ENL}]
\label{conj:ENL}
For every integer $k$, there exists an integer $f_k$, such that if\/ $\chi(G)\ge f_k$, then $\vec{\chi}(G)\ge k$.
\end{conj}

Clearly, $f_1=1$ and $f_2=3$. But it is still an open question whether $f_3$ exists.

Let $\mathcal{I}(G)$ be the family of all independent vertex-sets of $G$. For $v\in V(G)$, let $\mathcal{I}(G,v)$ be the subfamily of all those independent sets that contain $v$. For each independent set $I$, consider a nonnegative real variable $x_I$. Let $\chi_f(G)$ be the minimum value of
$\sum_{I\in\mathcal{I}(G)} x_I$
subject to
\begin{equation}\begin{array}{ll}
\sum_{I\in\mathcal{I}(G,v)} x_I \ge 1 & \mbox{for each vertex $v$},\\[2mm]
x_I\ge 0 & \mbox{for every } I\in\mathcal{I}(G).\\
\end{array}
\end{equation}
It is easy to see that in any optimal solution of this linear program, we have $x_I\le1$ for every $I\in\mathcal{I}(G)$. If we request that the values $x_I$ are integers 0 or 1, we obtain an integer program, whose optimal solution is the chromatic number $\chi(G)$. By this correspondence, we call $\chi_f(G)$ the \emph{fractional chromatic number} of the graph $G$.

The dual of this linear program computes the \emph{fractional clique number} $\omega_f(G)$, a relaxation to the rationals of the integer concept of the clique number. That is, a weighting $w:V\to \RR^+$ of the vertex-set $V=V(G)$ into the set of non-negative real numbers such that the total weight assigned to any independent set is at most 1, and the value $w(V)=\sum_{v\in V} w(v)$ is maximum possible under these conditions.

We can also define the relaxation to the rationals of the integer concept of the independence number. For every weighting $w:V\to \RR^+$ of the vertices of $G$, for which the total weight is $w(V) = \sum_{v\in V} w(v) = |V|$, we consider the maximum weight $w(I)$ over all independent sets $I\in {\mathcal I}(G)$. The minimum of this quantity, taken over all weightings $w$, is denoted by $\alpha_f(G)$ and called the \emph{fractional independence number} of $G$.

By the linear programming duality, we have the following lemma (see, e.g. \cite{GoRo}).

\begin{lem}
$\chi_f(G)=\omega_f(G)$.
\end{lem}

Analogously, we define the \emph{fractional chromatic number} $\chi_f(D)$ of a digraph $D$ by using acyclic vertex-sets playing the role of ${\mathcal I}(G)$. The maximum of $\chi_f(D)$ taken over all orientations $D$ of an undirected graph $G$ is the \emph{fractional dichromatic number} $\vec{\chi}_f(G)$ of $G$.

The main result of this paper is the following fractional version of the Erd\H{o}s and Neumann-Lara Conjecture.

\begin{thm}
\label{THM:dcn&fcn}
If\/ $\chi_f(G)\ge t$, then
$$\vec{\chi}_f(G) \ge \frac{t}{4\log(2et^2)}.$$
\end{thm}

The logarithms used in Theorem \ref{THM:dcn&fcn} and throughout the rest of the paper are always taken with respect to base 2.
Note that, when compared with Conjecture \ref{conj:ENL}, this result uses a stronger assumption that the fractional chromatic number is large, but it also gives a stronger conclusion.

The proof of Theorem \ref{THM:dcn&fcn} occupies the whole Section \ref{sect:proofs}.

Erd\H{o}s and Neumann-Lara proved \cite{ENL} that $c_1 \frac{n}{\log n}\le\vec\chi(K_n)\le c_2\,\frac{n}{\log n}$ for some constants $0 < c_1 < c_2$. For large $n$, $c_1\sim \tfrac{1}{2}$ and $c_2\sim 1$.
This implies that the bound of Theorem \ref{THM:dcn&fcn} is best possible up to the multiplicative factor~$4$.

Theorem \ref{THM:dcn&fcn} implies validity of Conjecture \ref{conj:ENL} for graphs whose chromatic number is bounded in terms of their fractional chromatic number. It is therefore natural to ask if some graphs with bounded fractional chromatic number might provide counterexamples to the conjecture of Erd\H{o}s and Neumann-Lara. In Section \ref{sect:Kneser} we treat such examples, in particular Kneser graphs with bounded fractional chromatic number. Let us observe that every graph with fractional chromatic number $q$ has a homomorphism to a Kneser graph whose fractional chromatic number is also $q$. In view of our main Theorem \ref{THM:dcn&fcn}, Kneser graphs with bounded fractional chromatic number are really the milestones, for which the Erd\H{o}s and Neumann-Lara conjecture should be tested first. We show that Kneser graphs have large dichromatic number as long as their chromatic number is large enough, see Theorem \ref{thm:DX Kneser}. The lower bound given by the theorem is surprisingly large and thus gives a firm support towards Conjecture \ref{conj:ENL}.

\subsection*{Previous work}

Dichromatic number was formally introduced by Neumann-Lara in 1982 \cite{NL82}, but its first appearance can be found in a paper by Erd\H{o}s \cite{ENL} in 1979, where he discusses several results and conjectures proposed by himself and Neumann-Lara. The notion of digraph coloring (for the chromatic and circular chromatic number) by using acyclic sets as color classes was rediscovered by Mohar \cite{Mo03} and studied in a follow-up paper by Bokal et al. \cite{BFJKM2004}. A few years later, it became clear that this notion of the chromatic number, when restricted to tournaments, is tightly related to the Erd\H{o}s-Hajnal conjecture (\cite{Ch14}), see, e.g.\ \cite{BCCFLSSS}.

As pointed out by Paul Erd\H{o}s in \cite{ENL}, ``It is surprisingly difficult to determine $\DX(G)$ even for the simplest graphs.'' It is thus not surprising that no real progress concerning the problems of him and Neumann-Lara in \cite{ENL} has been made until today.

\section{Principal and sparse vertex-sets}
\label{sect:proofs}
\label{sec:proofs lemmas}

This section is devoted to the proof of Theorem \ref{THM:dcn&fcn}. We let $V=V(G)$ and $n=|V|$. Apparently, we can assume that $\chi_f(G)=t$ and that $t>4\log(2et^2)$, which implies that $t>56$. Since $\omega_f(G)=\chi_f(G)$, there exists a weight function $w:V\to\RR^+$, such that $w(V)=t$, and for any $I\in \mathcal{I}(G)$, $w(I)\le 1$.
Here and in the sequel we write $w(A) = \sum_{v\in A} w(v)$ for any vertex-set $A\subseteq V$, and call this value the \emph{weight\/} of $A$.

From now on we fix $w$ and assume that
the vertices of $V(G)$ are listed as $\VEC v1n$ in the non-increasing order of their weights,
i.e. $w(v_{i+1})\le w(v_i)$ for $i=1,\dots,n-1$. For any subset $X$ of $V$, we also rank the elements in $X$ according to the ordering of $V$, and we denote by $X_k$ the subset of the first $k$ elements in $X$. In particular, $V_k=\{\VEC v1k\}$. We extend this notion to all positive real numbers by setting
$X_s := X_{\lfloor s\rfloor}$.

 For $s\in {\mathbb R}^+$, a nonempty subset $X$ of a vertex-set $Y$ is said to be \emph{$s$-principal\/} in $Y$ if $X\subseteq Y_{s|X|}$. That is, if $X$ has size $m$, then all elements of $X$ are within the first $\lfloor sm\rfloor$ vertices in $Y$. On the other hand, a subset $X$ of $Y$ is \emph{$s$-sparse} in $Y$ if $X$ contains no $s$-principal subset in $Y$. When the hosting set $Y$ for $s$-principal or $s$-sparse is not specified, by default it is $V$.

The following property is clear from the definition of sparse sets.

\begin{clm}
Any subset of an $s$-sparse set in $Y$ is also $s$-sparse in $Y$.
\end{clm}

\begin{clm}
\label{clm:sparse characterization}
$X$ is an $s$-sparse subset in $Y$ if and only if $|Y_k\cap X|< k/s$ for $1\le k\le |Y|$. In particular, if $X$ is $s$-sparse in $Y$, then $|X|<|Y|/s$.
\end{clm}

\begin{proof}
It is clear that $X$ is $s$-sparse if and only if for each $r=1,\dots,|X|$, $|Y_{sr}\cap X| < r$.
If $|Y_k\cap X|< k/s$ for every $k$, then this holds also for $k=\lfloor sr\rfloor$, implying that
$|Y_{sr}\cap X| < \lfloor sr\rfloor / s \le r$. Conversely, if $|Y_{sr}\cap X| < r$, then $|Y_{sr}\cap X| \le r-1$. Let $s(r-1) < k \le sr$. Then $|Y_k\cap X| \le |Y_{sr}\cap X| \le r-1 < k/s$. This proves the converse.
\end{proof}

The next claim about the total weight of an $s$-sparse set will be essential for us.

\begin{clm}\label{CLM:sparseislight}
If $X$ is an $s$-sparse subset of\/ $Y$, then $w(X)\le \frac{1}{s}\,w(Y)$.
\end{clm}

\begin{proof}
Let $\VEC y1r$ be the non-decreasing order of the elements of $Y$ with $r=|Y|$, and let $\VEC x1m$ be the ordering of $X$ with $m=|X|$. Since $X$ is an $s$-sparse subset of $Y$, we have $x_i\not\in Y_{si}$. Hence for $1\le i\le m$, $w(x_i)\le w(y_j)$ if $1\le j\le si$. Moreover, since $x_i\in Y\setminus Y_{si}$, we also have $w(x_i)\le w(y_j)$ for $j=\lceil si\rceil$.

For a real parameter $z\in(0,|Y|]$, define $f(z)=y_{\lceil z\rceil}$. Then $f(z)\ge w(x_1)$ for $0<z\le s$, $f(z)\ge w(x_2)$ for $s<z\le 2s$, \dots, $f(z)\ge w(x_m)$ for $(m-1)s < z \le ms$.
Therefore,
$$
    s\,w(X) = s\sum_{i=1}^m w(x_i) \le \int_0^{sm} f(z)dz \le \sum_{j=1}^{\lceil sm\rceil} w(y_j) \le w(Y),
$$
which gives what we were aiming to prove.
\end{proof}

The \emph{average degree} of a graph $G$, denoted by $\bar{d}(G)$, equals $\tfrac{1}{n}\sum_{v\in V(G)}d(v) = \tfrac{2e(G)}{n}$, where $e(G)$ denotes the number of edges in $G$. When $A$ is a vertex-set of $G$, the \emph{average degree} of $A$, denoted by $\bar{d}(A)$, means the average degree $\bar{d}(G[A])$ of the corresponding induced subgraph of $G$.

The following lemmas will be instrumental in proving Theorem \ref{THM:dcn&fcn}.

\begin{lem}\label{LEM:PDSC}
Suppose that $\chi_f(G)\ge t$. Let $d = 2\log(et^2)$ and suppose that $t\ge 2(d+1)$.
Then there exists an orientation of the edges of $G$ such that every $t$-principal set of $G$ with average degree at least $d$ contains a directed cycle.
\end{lem}

\begin{lem}\label{LEM:HWPDS}
Let $d = 2\log(et^2)$ and suppose that $t\ge 2(d+1)$.
If\/ $A$ is a vertex-set with weight more than $2d+4$ such that each independent subset of\/ $A$ has weight at most\/ $1$, then $A$ contains a $t$-principal subset whose average degree is at least $d$.
\end{lem}

Our main result, Theorem \ref{THM:dcn&fcn}, easily follows from Lemmas \ref{LEM:PDSC} and \ref{LEM:HWPDS}.

\begin{proof}[Proof of Theorem~\ref{THM:dcn&fcn}]
Let $d$ be as in Lemmas \ref{LEM:PDSC} and \ref{LEM:HWPDS}.
We may assume that $t > 4\log(2et^2) = 2d+4$.
By our assumption on the weight function of $G$ being the one that gives the fractional independence number of $G$, there is no independent set with weight more than $1$. Hence by Lemma~\ref{LEM:HWPDS}, every vertex-set with weight more than $2d+4$ contains a $t$-principal set with average degree at least $d$. By Lemma~\ref{LEM:PDSC}, there exists an orientation $D$, such that every $t$-principal set with average degree at least $d$ is cyclic. Hence every vertex-set with weight more than $2d+4$ is cyclic. Since the total weight for $V(G)$ is $t$, and every acyclic set in $D$ has weight at most $2d+4$, we must have $\chi_f(D)\ge\frac{t}{2d+4}$.
Therefore $\vec{\chi}_f(G) \ge \frac{t}{2d+4}$, which is the same as the bound in the theorem.
\end{proof}

First we give a proof of Lemma~\ref{LEM:HWPDS}.

\begin{proof}[Proof of Lemma~\ref{LEM:HWPDS}]
Suppose $A$ is a vertex-set whose weight is more than $2d+4$ and suppose that $A$ contains no $t$-principal subsets with average degree at least $d$. For a vertex $v \in A$, let $d^>(v)$ be the number of neighbors of $v$ in $G[A]$ that appear before $v$ in the ordering $\VEC v1n$. We have
$e(A_i)=\sum_{v\in A_i} d^>(v)$ for each $i=1,\dots,|A|$. Let $L=\{v\in A: d^>(v)\ge d\}$, and assume $L=\{v_{i_1},\dots,v_{i_l}\}$, where $l=|L|$. Then $e(V_{i_j}\cap A)\ge dj $ for $1\le j\le l$. In particular, if $|V_{i_j}\cap A|\le 2j$, then the average degree $\bar d(V_{i_j}\cap A)$ will be at least $d$. Since $A$ does not contain any $t$-principal set with average degree at least $d$, the set $V_{i_j}\cap A$ would not be $t$-principal in this case. Thus, we have one of the following:

(1) $|V_{i_j}\cap A|> 2j$, or

(2) $|V_{i_j}|> t|V_{i_j}\cap A|$.

\smallskip
\noindent
Let $L_1=\{v_{i_j}\in L:|V_{i_j}\cap A|> 2j\}$, and let $L_2=\{v_{i_j}\in L: |V_{i_j}|\ge t|V_{i_j}\cap A|\}$. Then $L=L_1\cup L_2$.
If $v_{i_j}\in L_1$, then it is not among the first $2j$ elements of $A$. As the $j$-th element in $L_1$ does not appear before $v_{i_j}$, the $j$th element of $L_1$ is not among the first $2j$ elements of $A$. Therefore $L_1$ is a 2-sparse subset of $A$. By Claim~\ref{CLM:sparseislight}, $w(L_1)\le \frac{1}{2}w(A)$.

For each $v_{i_j}\in L_2$, we have $|V_{i_j}\cap A|<\frac{1}{t}|V_{i_j}|$. As $|V_{i_j}\cap L_2|\le |V_{i_j}\cap A|<\frac{1}{t}|V_{i_j}|$,
we see by Claim \ref{clm:sparse characterization} that $L_2$ is a $t$-sparse subset of $V$. By Claim~\ref{CLM:sparseislight}, $w(L_2)\le \frac{1}{t}w(V)=1$.

Let $S=A\setminus L = \{v\in A: d^>(v) < d\}$. Then we have $w(S)\ge w(A)-w(L_1)-w(L_2)\ge \frac{w(A)}{2}-1$. Also, $G[S]$ is a $\lfloor d \rfloor$-degenerate graph, hence $S$ is $\lfloor d+1\rfloor$-colorable. Therefore, there is at least one independent set with weight at least
$$\frac{w(S)}{\lfloor d+1\rfloor}\ge \frac{w(A)/2-1}{d+1} > 1.$$
In the last inequality we used the fact that $w(A)>2d+4$. This completes the proof.
\end{proof}

In the rest of this section, we are going to use probabilistic method to prove Lemma~\ref{LEM:PDSC}, that there exists an orientation such that no $t$-principal set of $G$ with average degree at least $d$ is acyclic.
We will need the following lemma about acyclic orientations.

\begin{lem}
\label{lem:number acyc ori degrees}
The number of acyclic orientations of $G$ is at most
$$\prod_{v\in V(G)} (d_G(v)+1).$$
\end{lem}

\begin{proof}
Let $v_1, \dots, v_n$ be the vertices of $G$. For any orientation $D$ of $G$, we consider its out-degree sequence $(d_D^+(v_1), d_D^+(v_2), \dots, d_D^+(v_n))$. Since $d_D^+(v)\le d_G(v)$, there are at most $\prod_{v\in V(G)} (d_G(v)+1)$ different out-degree sequences. On the other hand, for any out-degree sequence corresponding to some acyclic orientation, we can determine the orientation inductively, starting from that there is a vertex with out-degree $0$, and all the edges incident to it are oriented towards it. So given any two acyclic orientations, it is easy to see that their out-degree sequence cannot be the same. Therefore there are at most $\prod_{v\in V(G)} (d_G(v)+1)$ acyclic orientations.
\end{proof}

We pick an orientation of the graph $G$ among all the $2^{e(G)}$ orientations uniformly at random. Then for each edge, the probability for each direction we pick is $1/2$, and it is mutually independent from the orientation of other edges. We have the following corollary of Lemma \ref{lem:number acyc ori degrees}.

\begin{cor}\label{COR:acyclic}
The probability that a random orientation of $G$ is acyclic is less than $2^{-e(G)(1-2\alpha)}$,
where $\alpha = \log(\bar d(G) + 1) / \bar d(G)$.
\end{cor}

\begin{proof}
Let $f(x)=\ln(x+1)$. Since $f'(x)=1/(x+1)$ and $f''(x)=-1/(x+1)^2<0$, we see that $f(x)$ is concave for $x>-1$. Therefore, $\sum_{v\in V(G)} \ln(d_G(v)+1)\le n\ln(\bar{d}(G)+1)$. That is
$$
  \prod_{v\in V(G)} (d_G(v)+1)\le {(\bar{d}(G)+1)^n} = 2^{\alpha n \bar d(G)} = 2^{2\alpha e(G)}.
$$
Since there are $2^{e(G)}$ different orientations of $G$, among which at most $(\bar{d}(G)+1)^n$ are acyclic, the probability that a random orientation of $G$ is acyclic is at most
$2^{2\alpha e(G)}\,2^{-e(G)} = 2^{-e(G)(1-2\alpha)}$.
\end{proof}

We are going to use the following fact.

\begin{fact}
${\lfloor tk \rfloor \choose k} < (et)^k$.
\end{fact}

\begin{proof}
By Stirling's Formula, $k!\ge \sqrt{2\pi k}\,(\frac ke)^k > (\frac ke)^k$. Therefore we have
$${\lfloor tk \rfloor \choose k} \le \frac{tk(tk-1)\cdots(tk-k+1)}{k!} < \frac{(tk)^k}{(\frac ke)^k}=(et)^k.$$
\end{proof}

Now we are ready to prove Lemma~\ref{LEM:PDSC}.

\medskip

\begin{proof}[Proof of Lemma~\ref{LEM:PDSC}]
Let $A$ be a $t$-principal vertex-set of cardinality $k$ and with average degree at least $d$. By Corollary~\ref{COR:acyclic} applied to the graph $G[A]$, the probability that a random orientation of $G$ is acyclic on $A$ is small:
\begin{equation}
  Pr(A\mbox{ is ayclic})\le 2^{-e(G[A])(1-2\alpha)} = 2^{-k\bar{d}(A)(1-2\alpha)/2},
  \label{eq:Pr acyclic L2.1}
\end{equation}
where $\alpha = \log(\bar d(A) + 1) / \bar d(A)$. Since $\bar d(A) \ge d$, we also have the following inequality:
\begin{equation}
  \tfrac{1}{2}\,k\bar{d}(A)(1-2\alpha) \ge \tfrac{1}{2}\,kd - k \log(d+1).
  \label{eq:d and bar d}
\end{equation}

As a $t$-principal set, $A$ is a $k$-set contained in $V_{kt}$. Thus, there are at most ${\lfloor tk \rfloor \choose k}$ $t$-principal sets with $k$ elements. By (\ref{eq:Pr acyclic L2.1}) and (\ref{eq:d and bar d}), the probability that some $t$-principal $k$-set with average degree at least $d$ is acyclic is at most
\begin{align*}
   2^{-kd/2 + k \log(d+1)}{\lfloor tk \rfloor \choose k}
   & < 2^{-kd/2 + k \log(d+1)}(et)^k \\
   & = \bigl(2^{-d/2 + \log(d+1)}\, et \bigr)^k \\
   & = \Bigl(\frac{d+1}{t}\Bigr)^k \le 2^{-k}.
\end{align*}
Therefore, the probability that some $t$-principal set with average degree at least $d$ is acyclic is at most
$$
\sum_{k=1}^n Pr(\exists \mbox{ an acyclic $t$-principal $k$-set $A$ with } \bar{d}(A)\ge d) < \sum_{k=1}^n 2^{-k} < 1.
$$
This implies that there exists an orientation such that no $t$-principal sets of $G$ with average degree at least $d$ is acyclic.
\end{proof}

In the last summation in the above proof, we could also use the fact that $A$ needs to have average degree at least $d$, hence $k=|A|\ge d+1$. Then the probability is less than $2^{-d}(1/2+1/4+\cdots) = 2^{-d} = e^{-2}t^{-4}$. This shows that most orientations give the conclusion of the lemma.

\section{Blow-ups and Kneser graphs}
\label{sect:Kneser}

Theorem \ref{THM:dcn&fcn} raises a question whether the real condition in Conjecture \ref{conj:ENL} should be about the fractional chromatic number instead of the usual chromatic number. It also shows that any possible infinite family of graphs providing a counterexample to the conjecture would have bounded fractional chromatic number (and arbitrarily large chromatic number). Any graph $G$ with $\chi_f(G)\le q$ has a homomorphism into a Kneser graph $K(n,k)$ with $n/k\le q$ (see, e.g., \cite{GoRo}), and thus it is instrumental to check the validity of the Erd\H{o}s and Neumann-Lara conjecture for Kneser graphs. We do this in this section. The results may be of independent interest because of the tools involved in the proofs.

Given a graph $H$, the \emph{blow-up} of $H$ with \emph{power} $m$, denoted by $H^{(m)}$, is the graph obtained from $H$ by replacing each vertex by an independent set of size $m$ (called the \emph{blow-up\/} of the vertex), and for each edge $xy$ in $H$, the two blow-ups of $x$ and $y$ form a complete bipartite graph $K_{m,m}$. The subgraph of $H^{(m)}$ replacing an edge of $H$ is isomorphic to $K_{m,m}$ and will be referred to as the \emph{blow-up} of that edge.

\begin{lem}\label{LEM:BUCOL}
Let $k$ be a positive integer and let $H$ be a graph with $\chi(H)>k$. If there is an orientation $D$ of $H^{(m)}$, such that in the blow-up of any edge, no subgraph isomorphic to $K_{\lceil\frac mk\rceil, \lceil\frac mk\rceil}$ is acyclic, then $\chi(D)>k$.
\end{lem}

\begin{proof}
Suppose that $D$ is an orientation of $H^{(m)}$ with the stated property, and that $D$ is $k$-colorable. We are going to define a $k$-coloring of $H$ as follows. For each vertex $x$ of $H$, let $X$ be the blow-up of $x$. There is a color $c=c(x)$ that is used for at least $\lceil m/k\rceil$ vertices of $X$ in the $k$-coloring of $D$. Consider an edge $xy\in E(H)$. If $c(x)=c(y)$, let $A$ be the set of vertices in the blow-up of $x$ having color $c(x)$, and let $B$ be the set of vertices in the blow-up of $y$ colored $c(y)$. Then $|A|, |B|\ge \lceil m/k\rceil$. By assumption, $A\cup B$ is not acyclic in $D$. This contradicts the fact that all vertices in $A\cup B$ receive the same color. We conclude that $c(x)\ne c(y)$. This shows that the function $c:V(H)\to [k]$ is a $k$-coloring of $H$. This contradiction completes the proof.
\end{proof}

\begin{lem}\label{LEM:BUCY}
If\/ $2+2\log m\le \lceil m/k\rceil$, and $\chi(H)>k$, then $\DX(H^{(m)})> k$.
\end{lem}

\begin{proof}
Let $r = \lceil m/k\rceil$.
By Lemma~\ref{LEM:BUCOL}, it suffices to prove that for $m\ge 4k\log(ke)$, there is an orientation of $H^{(m)}$, such that no copy of $K_{r,r}$ in the blow-up of any edge is acyclic.
It suffices to show that the blow-up of each edge of $H$ has such an orientation, and we will prove that a random orientation will have this property with positive probability.

Let us now consider a random orientation of the blow-up $K\cong K_{m,m}$ of a fixed edge of $H$.
As there are $(2r)!$ permutations of the vertices of $K_{r,r}$, and each acyclic orientation is corresponding to at least one such permutation, the probability that a copy of $K_{r,r}$ in $K$ is acyclic is at most $(2r)!/2^{r^2}$. As there are  ${m\choose r}^2$ copies of $K_{r, r}$ in $K$, the probability that some copy of $K_{r,r}$ is acyclic is at most
$$
  (2r)!\,2^{-r^2}{m\choose r}^2 = {2r\choose r}\left(\frac{m!}{(m-r)!}\right)^2 2^{-r^2} \le
  {2r\choose r} m^{2r} 2^{-r^2} < 2^{-r^2+2r} m^{2r}.
$$
The value on the right-hand side of this inequality is at most 1 if $2+2\log m \le r$. Thus, there is at least one orientation for $K$, such that no copy of $K_{r,r}$ in $K$ is acyclic. This completes the proof.
\end{proof}

A \emph{Kneser graph}, denoted by $KG(n,k)$, is a graph whose vertices correspond to the $k$-element subsets of a set with $n$ elements, and two vertices are adjacent if their corresponding sets are disjoint.
In 1978, Lov\'{a}sz \cite{Lo79} proved that the chromatic number of $KG(n,k)$ is $n-2k+2$, which was conjectured by Kneser in 1955. It is also known that the fractional chromatic number of $KG(n,k)$ is $\frac{n}{k}$.
Thus, Kneser graphs provide examples of graphs, whose fractional chromatic number can be bounded, while their chromatic number would be arbitrarily large.

\begin{thm}\label{THM:KGBU}
Let $n,k,t$, and $x$ be nonnegative integers such that $0<k<n$ and $x<kt$.
The Kneser graph\/ $KG(nt,kt-x)$ contains the blow-up of $KG(n,k)$ with power ${k(t-1)\choose x}$ as a subgraph. Furthermore, when $x<t$, it contains the blow-up of $KG(n,k)$ with power ${kt\choose x}$, and when $x=t$, it contains the blow-up of $KG(n,k)$ with power ${kt\choose x}-k$.
\end{thm}

\begin{proof}
Let $G=KG(nt, kt-x)$ and $H=KG(n,k)$. We can consider the vertices of $G$ as the $(kt-x)$-subsets of $[n]\times [t]$, and the vertices of $H$ as the $k$-subsets of $[n]$. Given a vertex $A \in V(G)$, let $f(A)=\{a\in [n]: (a,b)\in A \mbox{ for some } b\in [t]\}$; i.e., $f$ is the projection from $[n]\times [t]$ to $[n]$.

Notice that for two vertices $A$ and $B$ of $G$, if $f(A)\cap f(B)=\emptyset$, then $A\cap B=\emptyset$. In particular, if $|f(A)|=|f(B)|=k$ and $f(A)$ and $f(B)$ are adjacent in $H$, then $A$ and $B$ are adjacent in $G$.
For every vertex $X\in V(H)$, $X$ is a $k$-set. If $f(A)=X$ for some vertex in $G$, then $A$ is a $(kt-x)$-subset of $X\times [t]$.

If $x<t$, then for every $(kt-x)$-subset  $A$ of $X\times [t]$, we have $f(A)=X$. So there are ${kt\choose x}$ vertices in $G$ that map to $X$ by $f$. This shows that $G$ contains a  blow-up of $H$ with power ${kt\choose x}$.

If $x=t$, among the $(kt-x)$-subsets of $X\times [t]$, only $k$ of them have $f(A)\not= X$, and each misses one element of $X$. Therefore, $G$ contains a blow-up of $H$ with power ${kt\choose x}-k$.

In general, we can consider those $(kt-x)$-subsets of $X\times [t]$, which contain all the elements in $X\times \{1\}$. There are ${k(t-1)\choose x}$ such subsets, and each of them is mapped to $X$ by $f$. Therefore, $G$ contains a blow-up of $H$ with power ${k(t-1)\choose x}$.
\end{proof}

Erd\H{o}s and Neumann-Lara (see \cite{ENL}) proved that
\begin{equation}
    \DX(K_n)\ge \frac {n}{2\log(n)}.
\label{eq:ENL_Kn_1/2}
\end{equation}
This result can be extended to blow-ups of complete graphs as follows.

\begin{thm}\label{THM:BUCG}
$\DX(K_n^{(k)}) > \min\{\frac{nk}{4\log(nk)}, \frac n2\}$.
\end{thm}

\begin{proof}
Let $G = K_n^{(k)}$, and let $t=\max\{\lceil 4\log(nk)\rceil, 2k\}$. It suffices to show that there is an orientation of $G$ such that every vertex-set of size $t$ has a directed cycle.

Let $A\subseteq V(G)$ be a set of size $t$. As $G$ is a balanced complete $n$-partite graph and each part has size $k$, each vertex in $A$ has at least $t-k$ neighbors in $A$. Therefore, $G[A]$ has at least $\frac{t(t-k)}{2}$ edges. Among all orientations of $G[A]$, at most $t!$ of them are acylic.
Thus, $Pr(A \mbox{ is acyclic})\le t!\, 2^{-t(t-k)/2}$. There are ${nk\choose t}$ such set. It follows that with a random orientation of $G$ we have the following:
\begin{eqnarray}
Pr(\exists A: |A|=t \mbox{ and } A \mbox{ is acyclic}) &\le& {nk\choose t}\,t!\, 2^{-t(t-k)/2}\nonumber\\
&<& \bigl(nk 2^{-(t-k)/2}\bigr)^t.\label{eq:Pr acyclic}
\end{eqnarray}
Since $t=\max\{\lceil 4\log(nk)\rceil, 2k\}$, we have $t-k\ge 2\log(nk)$. Therefore, $nk 2^{-(t-k)/2} \le 1$.
This implies that the value on the right-hand side of (\ref{eq:Pr acyclic}) is at most 1. Consequently,
there is at least one orientation $D$ of $G$ such that any set of size $t$ has a directed cycle. Therefore, in a proper coloring of $D$, each color class has size at most $t-1$. Hence $\DX(K_n^{(k)})\ge \chi(D)\ge \frac{nk}{t-1} > \min\{\frac{nk}{4\log(nk)}, \frac n2\}$.
\end{proof}

\begin{thm}
\label{thm:DX Kneser}
For any positive integers $n,k$ with $n\ge 2k$, we have
$$\DX(KG(n,k)) \ge \left\lfloor \frac{n-2k+2}{8\log(n/k)} \right\rfloor.$$
\end{thm}

\begin{proof}
Let $G=KG(n,k)$. Observe that $\DX(G)=1$ only when the graph is a forest. In our case this happens only when $n=2k$, and the inequality holds in this case. Let $\di{z=\left\lfloor\frac{\chi(G)}{8\log(\chi_f(G))}\right\rfloor=\left\lfloor\frac{n-2k+2}{8\log(n/k)}\right\rfloor}$. We just need to consider the case when $z\ge 2$. Hence, we may assume $\chi(G)=n-2k+2\ge 16$.
For $k\le3$, $G$ contains the complete graph on $\lfloor\frac nk \rfloor$ vertices. By using (\ref{eq:ENL_Kn_1/2}), it is easy to see that the theorem holds in this case. Thus, we may assume that $k\ge4$.

Let $H_1=KG(\lfloor\frac{n}{2k}\rfloor,1)=K_{\lfloor\frac n{2k}\rfloor}$. Theorem~\ref{THM:KGBU} with $t=2k$ and $x=k$ implies that $KG(2k\lfloor\frac{n}{2k}\rfloor,2k-k)$ contains the blow-up of $H_1$ with power ${2k\choose k}$. Since $n\ge 2k\lfloor\frac{n}{2k}\rfloor =: n'$, $G$ contains $KG(n',k)$, and hence also the blow-up of $H_1$ with power ${2k\choose k}$. By Theorem~\ref{THM:BUCG}, we have
\begin{equation}
  \DX(G) > \min\left\{\frac{1}{2}\left\lfloor\frac n{2k}\right\rfloor, \frac{{2k\choose k}\left\lfloor\frac n{2k}\right\rfloor}{4\log\left({2k\choose k}\lfloor\frac n{2k}\rfloor\right)}\right\}.
  \label{eq:H1 blow-up}
\end{equation}

Suppose first that $k\le 2\log (n/k)$. If the minimum in (\ref{eq:H1 blow-up}) is attained with the first term, then we have $\DX(G) > \frac{1}{2}\left\lfloor\frac n{2k}\right\rfloor \ge \frac{n-2k+1}{4k}$. Since $\DX(G)$ is an integer, we conclude that
$\DX(G) \ge \frac{n-2k+2}{4k} \ge \frac{n-2k+2}{8\log(n/k)} \ge z$ as claimed.
If the minimum is attained with the second term, then noting that ${2k\choose k}\ge 4k$, we see that
$$
  \DX(G) \ge
  \frac{4k\left\lfloor\frac n{2k}\right\rfloor}{4\log\left(4k\lfloor\frac n{2k}\rfloor\right)} \ge
  \frac{n}{4\log(2n)} = \frac{n}{4\log(n/k)+4\log(k/2)} \ge \frac{n}{8\log(n/k)} \ge z.
$$

Suppose now that $k>2\log (n/k)$, that is $n<k2^{k/2}$. To simplify notation, we also set $r := \lfloor \frac{k}{2} \rfloor$. We will distinguish two cases.

Suppose first that $n\ge 4k$. The result is clear if $z\le2$. Otherwise, $z\ge 3$ and we have $n-2k+2 \ge 24\log(n/k)\ge 48$. This contradicts the assumption that $n < k 2^{k/2}$ if $k\le6$. If $k=7$, then $n\ge 60$ and hence $\frac{n}{k}\ge 8$. In particular, $G$ contains $K_7$ as a subgraph. Since $\DX(K_7)=3$ \cite{NL94}, we may assume that $z\ge 4$; otherwise the theorem is proved. However, if $z\ge4$ and $\frac{n}{k}\ge 8$, then $n-2k+2\ge 32\log(n/k) \ge 96$. Again, this contradicts the assumption that $n < k 2^{k/2}$. We need to do one more special case, when $k=9$. In this case, we conclude as above, first that $n\ge64$, next that $G$ contains $K_7$ and that we may assume that $z\ge4$. This implies that $n\ge 16 + 32 \log(n/9)$. The smallest $n$ for which this inequality is satisfied is $128$. Therefore $G$ contains the complete graph $K_{14}$ as a subgraph. It is known \cite{NL94} that $\DX(K_{11})=4$, thus we may assume henceforth that $z\ge5$. In this case, the inequality for $z$ becomes $n\ge 16 + 40\log(n/9)$. The smallest $n$ for which this is satisfied is $193$. Thus, $193 \le n\le 203 = \lfloor k 2^{k/2}\rfloor$. In this case we may use (\ref{eq:H1 blow-up}) which shows that $\DX(G)>5$. The assumption that $z\ge 6$ now gives a contradiction to the fact that $n\le 203$.

From now on we may assume that either $k=8$ or $k\ge 10$.
Let $H_2=KG(\lfloor\frac n3\rfloor, r)$. By Theorem~\ref{THM:KGBU} with $t=3$ and $x=r$ (when $k$ is even) or $x=r-1$ (when $k$ is odd), we conclude that $G$ contains a blow-up of $H_2$ with power $m = {2r \choose x}$.

Our goal here is to apply Lemma~\ref{LEM:BUCY}, so we need to argue that $2+2\log m\le \lceil\frac mz\rceil$. We can estimate $z$ as follows:
$$
  z\le \frac{n}{8\log(n/k)} < \frac{k2^{k/2}}{8\log(\frac{2^{k/2}k}k)}=2^{\frac k2 - 2}.
$$
Thus, it suffices to see that
\begin{equation}
    2+2\log m \le \bigl\lceil\, m/\lfloor 2^{\frac k2 - 2} \rfloor \,\bigr\rceil.
\label{eq:needed for Lemma 3.2}
\end{equation}

Inequality (\ref{eq:needed for Lemma 3.2}) can be proved by induction on $k$ as follows. First of all, the inequality holds for $k=8$ and $k=11$ (base cases), which the reader can easily verify. Next, we observe that when we increase $k$ by 2, the left-hand side increases by at most 4. On the other hand, the right-hand side increases by more than 4 when $k\ge8$.

This shows that we can apply Lemma~\ref{LEM:BUCY} when $k=8$ or $k\ge10$.
We do this for every $n$ such that $4k\le n < k 2^{k/2}$. Since $\chi(H_2)=\lfloor\frac {n}3\rfloor-2\lfloor \frac k2\rfloor+2\ge \frac{n-2k+2}{16}\ge \frac{\chi(G)}{8\log(n/k)}\ge z$, we have by Lemma~\ref{LEM:BUCY} that
$\DX(G)\ge z$.

For the remaining subcase, we assume that $n<4k$. Let $H_3=KG(\lfloor\frac n2\rfloor, r+2)$. By Theorem~\ref{THM:KGBU} with $t=2$ (and $x=k-2r+4\ge4$), we see that $G$ contains a blow-up of $H_3$ with power ${r+2\choose 4}$. Observe that
$$
 z = \Bigl\lfloor\frac{n-2k+2}{8\log(n/k)}\Bigr\rfloor \le
 \Bigl\lfloor\frac {4k-2k+2}{8\log (4k/k)}\Bigr\rfloor =
 \Bigl\lfloor\frac{k+1}8\Bigr\rfloor.
$$
By induction on $k$, we easily prove the following:
$$
 2 + 2\log{{r+2}\choose 4} \le
 \Bigl\lceil\, \frac{{{r+2}\choose 4}}{\lfloor \frac{k+1}{8}\rfloor}\,\Bigr\rceil \le
 \Bigl\lceil\, \frac{{{r+2}\choose 4}}{z}\,\Bigr\rceil.
$$
This shows that we can apply Lemma~\ref{LEM:BUCY}. Recall that $n-2k+2\ge 16$. Then,
$\chi(H_3)=\lfloor \frac n2\rfloor - 2r - 2 \ge \lfloor\frac{n-2k+2}{8\log(n/k)}\rfloor = z$.
By Lemma~\ref{LEM:BUCY}, we conclude that $\DX(G)\ge z$.
\end{proof}

Let $n=2k+z-2$ with $k\gg z$, and let $G=KG(n,k)$. Then $\chi(G) = z$ and $\chi_f(G) = \frac{n}{k} < 2+\frac{z}{k}$. By Theorem \ref{thm:DX Kneser}, $\DX(G) \ge \frac{z}{8\log(n/k)} > \frac{z}{16}$.
Thus, we have the following result.

\begin{cor}
For any $\varepsilon>0$ and any integer $t$, there is a graph $G$ with $\chi_f(G)<2+\varepsilon$ and $\DX(G)\ge t$.
\end{cor}

So $\DX(G)$ is lower bounded by a function of $\chi_f(G)$ but not upper bounded by a function of $\chi_f(G)$.

\end{document}